\documentclass{gtpart}

\usepackage[all]{xy}
\usepackage{enumerate}

\title[Deriving Deligne--Mumford Stacks]{Deriving Deligne--Mumford Stacks with\\Perfect Obstruction 
  Theories}

\author[T Sch\"urg]{Timo Sch\"urg}
\givenname{Timo}
\surname{Sch\"urg}
\address{Mathematisches Institut \\Universit\"at Bonn \\Endenicher Allee 60\\53115 Bonn\\Germany}
\email{timo\_ schuerg@operamail.com}
\urladdr{}


\arxivreference{1005.3945}
\arxivpassword{kqkgp}

\volumenumber{}
\issuenumber{}
\publicationyear{}
\papernumber{}
\startpage{}
\endpage{}
\doi{}
\MR{}
\Zbl{}
\received{}
\revised{}
\accepted{}
\published{}
\publishedonline{}
\proposed{}
\seconded{}
\corresponding{}
\editor{}
\version{}

\DeclareMathOperator{\Shv}{Shv}
\DeclareMathOperator{\Sp}{Sp}
\DeclareMathOperator{\cofib}{cofib}
\DeclareMathOperator{\fib}{fib}
\DeclareMathOperator{\CAlg}{CAlg}

\DeclareMathOperator{\Sch}{Sch}
\DeclareMathOperator{\Spec}{Spec}

\DeclareMathOperator{\Ext}{Ext}

\DeclareMathOperator{\Tor}{Tor}
\DeclareMathOperator{\et}{\acute{e}t}

\DeclareMathOperator{\id}{id}

\DeclareMathOperator{\QCoh}{QCoh}

\DeclareMathOperator{\sm}{sm}

\DeclareMathOperator{\Der}{Der}

\DeclareMathOperator{\End}{End}

\DeclareMathOperator{\con}{con}

\DeclareMathOperator{\Mod}{Mod}
\DeclareMathOperator{\Fun}{Fun}

\newcommand{\Cat}[1]{\mathcal{#1}}

\newcommand{\Sh}[1]{\mathcal{#1}}

\newcommand{\dS}[1]{\mathfrak{#1}}
\newcommand{\iT}[1]{(\mathcal{#1},\mathcal{O}_{\mathcal{#1}})}
\newcommand{\Oof}[1]{\mathcal{O}_{\mathcal{#1}}}
\newcommand{\trunc}[2]{\tau _{\leq #1} #2}
\renewcommand{\theta}{\vartheta}

\newcommand{\IE}{\mathbb{E}}

\newcommand{\IQ}{\mathbb{Q}}

\newtheorem{thm}{Theorem}[section]
\newtheorem{lem}{Lemma}[section]
\newtheorem{prop}{Proposition}[section]
\newtheorem{cor}{Corollary}[section]

\theoremstyle{definition}
\newtheorem{defn}{Definition}[section]
\newtheorem{exmp}{Example}[section]
\newtheorem{ass}{Assumption}[section]

\theoremstyle{remark}
\newtheorem{rem}{Remark}[section]

\makeautorefname{ass}{Assumption}
\makeautorefname{defn}{Definition}

\makeatletter
  \let\c@lem=\c@thm
\makeatother
\makeatletter
  \let\c@prop=\c@thm
\makeatother
\makeatletter
  \let\c@cor=\c@thm
\makeatother
\makeatletter
  \let\c@defn=\c@thm
\makeatother
\makeatletter
  \let\c@exmp=\c@thm
\makeatother
\makeatletter
  \let\c@ass=\c@thm
\makeatother
\makeatletter
  \let\c@rem=\c@thm
\makeatother

\begin{document}
\begin{abstract}
  We give conditions for a $n$--connective quasi-coherent obstruction 
  theory on a Deligne--Mumford stack to come from the structure of a 
  connective spectral Deligne--Mumford stack on the underlying topos.  
  Working over a base ring containing the rationals, we obtain the 
  corresponding result for derived Deligne--Mumford stacks.
\end{abstract}
\begin{asciiabstract}
  We give conditions for a n-connective quasi-coherent obstruction 
  theory on a Deligne-Mumford stack to come from the structure of a 
  connective spectral Deligne-Mumford stack on the underlying topos.  
  Working over a base ring containing the rationals, we obtain the 
  corresponding result for derived Deligne-Mumford stacks.
\end{asciiabstract}
\maketitle
\section{Introduction}
Some moduli spaces playing an important role in enumerative geometry 
carry an additional structure. Apart from the cotangent complex, which 
controls deformations and obstructions of the objects parametrized by 
the moduli space, there sometimes exist another complex doing the same 
job.  If the moduli space in question is very singular, the cotangent 
complex will have cohomology in arbitrary many degrees. In many cases, 
the replacement complex has much better finiteness properties, being 
locally isomorphic to a finite complex of vector bundles. If the 
replacement complex is perfect and of Tor--amplitude $\leq 1$ it gives 
rise to the virtual fundamental class of Li--Tian \cite{li98} and 
Behrend--Fantechi \cite{behrend97}, which is the key to actually 
producing numbers.

Ever since this phenomenon was observed by Kontsevich 
\cite{kontsevich95}, it was suspected that the replacement complex is a 
shadow of a derived structure on the moduli space. In the meantime, the 
foundations of derived algebraic geometry have been firmly laid out by 
To\"en--Vezzosi  \cite{toen08} and Lurie \cite{lurie04}. Using these 
theories, in many examples derived moduli spaces having the `correct' 
cotangent complex have been found (for examples see To\"en's overview 
\cite{toen09}). These derived enhancements have the same underlying 
topological space as their classical counterpart, the derived structure 
just being a nilpotent thickening of the structure sheaf. In 
\cite[Section  4.4.3]{toen09}, To\"en observed that such a derived 
enhancement automatically induces a replacement complex for the 
cotangent complex of the classical moduli space. The replacement is 
simply the cotangent complex of the derived enhancement, which might 
well be very different from the cotangent complex of the classical part 
and enjoy much better finiteness properties.

Using the approach of To\"en mentioned above, we can regard the 
replacement complex as the cotangent complex of some possible derived 
enhancement. We thus already know quite a lot about a possible derived 
enhancement inducing the replacement complex: we know its underlying 
space and cotangent complex. The problem then is to lift this 
information on the tangent level to an actual derived structure sheaf on 
the space.  

This formulation makes the problem tractable to using obstruction 
calculus to find a possible derived structure on the moduli space in 
question inducing the replacement complex. In a certain sense, this 
defeats the purpose of derived algebraic geometry, as part of the 
motivation for derived algebraic geometry was precisely avoiding such 
calculations and simply writing down the functor the lifted moduli 
problem should represent\footnote{see the overviews by To\"en--Vezzosi 
  \cite[Section 6]{toen03} and Lurie \cite[page 9]{Elliptic}}. 

Behrend and Fantechi in \cite{behrend97} axiomatized the phenomenon of a 
replacement complex for the cotangent complex of a moduli space to the 
notion of an \emph{obstruction theory}. The main theorem of this paper 
is an obstruction calculus that gives necessary and sufficient 
conditions for an obstruction theory of arbitrary length on a 
Deligne--Mumford stack to come from the structure of a spectral 
Deligne--Mumford stack on the same underlying topos 
(\fullref{thm:geomresult}). Over a base ring containing the rationals 
the theory of spectral Deligne--Mumford stacks and derived 
Deligne--Mumford stacks coincide. Adding this extra assumption, we 
obtain the corresponding calculus for derived Deligne--Mumford stacks.

Recall that an obstruction theory is called \emph{perfect} if the 
replacement complex is in fact a perfect complex. If we assume the given 
obstruction theory to be perfect we can give a more precise description 
of the obstruction classes that measure if a perfect obstruction theory 
is induced by a derived structure.
%

The method of proof uses obstruction calculus for nilpotent thickenings 
of derived rings. The basic observation underlying the whole work is 
easily described.  An obstruction theory for a commutative ring $A$ is 
given by a morphism $\phi \colon E \to L_A$, where $E$ is a complex of 
$A$--modules and $L_A$ is the cotangent complex. Such a morphism can 
always be completed to a cofiber sequence
\[
  E \overset{\phi}{\longrightarrow} L_A \overset{\eta}{\longrightarrow} 
  K.
\]
Now the datum of a morphism $\eta \colon L_A \to K$ defines a 
square-zero extension $A^{\eta} \to A$. The cotangent complex of 
$A^{\eta}$ is already an excellent approximation of $E$. There exists a 
comparison map $E \to L_{A^{\eta}}$ which is an equivalence in low 
degrees. The remaining work then is to find further square-zero 
extensions of $A^{\eta}$ that successively correct the difference in 
higher degrees. This is only possible if the obstruction theory lifts to 
the nilpotent thickening $A^{\eta}$. If this is possible, this process 
will allow us to lift the structure sheaf of the classical part step by 
step to a structure sheaf of derived rings which has the right cotangent 
complex.  The main advantage of this approach is that it is global from 
the start, thus avoiding all gluing issues. The practical value of such 
a result is small though. Given a moduli problem equipped with a 
obstruction theory it is far better to find the appropriate derived 
formulation of the moduli problem, as the true derived moduli space 
contains much more information than just the induced perfect obstruction 
theory on the truncation.

\subsection*{Conventions}
\begin{itemize}
  \item Given a stable $\infty$--category $\Cat{C}$ equipped with a 
    $t$--structure in the sense of Lurie's treatise on Higher Algebra 
    $\cite{HigherAlgebra}$, an object $X$ of $\Cat{C}$ is said to be 
    \emph{$n$--connective} if $X \in \Cat{C}_{\geq n}$. A morphism $f 
    \colon X \to Y$ is $n$--connective it its fiber $\fib(f)$ is 
    $n$--connective.
  \item Given a commutative ring $A$ we will denote by $\Mod_A$ the 
    $\infty$--category of $A$--module spectra. This category contains the 
    category of ordinary $A$--modules as heart of a $t$--structure.  
    Roughly, objects of $\Mod_A$ consist of possibly unbounded chain 
    complexes of ordinary $A$--modules.
\end{itemize}

\subsection*{Acknowledgments}
I would like to thank Manfred Lehn, Marc Nieper-Wi\ss kirchen and 
Gabriele Vezzosi for countless helpful discussions on the subject,
Moritz Groth, Parker Lowrey and David Carchedi for many explanations 
about $\infty$--categories, and the referee for his/her suggestions.  
Finally I thank Barbara Fantechi for emphasizing the benefits of global 
constructions.  The idea to develop an obstruction calculus for the 
existence of a derived structure globally step by step is by her. The 
dependence on Lurie's volumes is obvious from the number of citations.

The published version of this paper contains as main result the 
erroneous claim that every obstruction theory is induced by a derived 
structure. I would like to thank Richard Thomas and the anonymous 
referee for pointing out the mistake and giving some hints on what 
assumptions had to be added.

This work was supported by the SFB/TR 45 `Periods, Moduli Spaces and 
Arithmetic of Algebraic Varieties' of the DFG (German Research 
Foundation).


\section{The Algebraic Case}

In this section we first treat the problem of when a derived structure 
induces an obstruction theory in an abstract setting. The abstract 
setting will be given by a stable symmetric monoidal $\infty$--category 
$\Cat{C}$ equipped with a $t$--structure satisfying some assumptions 
(\fullref{ass:t}). After reviewing some results on the cotangent complex 
of a commutative algebra object in such a category, we then define the 
notion of an $n$--connective obstruction theory on a commutative algebra 
object $A \in \CAlg(\Cat{C})$ (\fullref{defn:pot}). We then define how a 
morphism $f \colon B \to A$ can induce a given obstruction theory. As 
main result we prove that for any given commutative algebra object $A$ 
with fixed obstruction theory there always exists a morphism $f \colon B 
\to A$ inducing the obstruction theory. 

The main example for $\Cat{C}$ we have in mind is the $\infty$--category 
of $k$--module spectra where $k$ is a commutative ring containing $\IQ$ 
(\fullref{exmp:overk}). Concrete models for connective commutative 
algebra objects in this category are given by simplicial $k$--algebras 
(\fullref{exmp:simplicial}) or by connective commutative differential 
graded algebras over $k$ (\fullref{exmp:cdga}). In this example we also 
show a finiteness result for any commutative algebra object $B$ inducing 
an obstruction theory on a finitely presented discrete commutative 
$k$--algebra in case the $n$--connective obstruction theory is perfect.

\subsection{Background}
Throughout this paper, the following assumption will be made with regard 
to the $\infty$--category $\Cat{C}$ in question.

\begin{ass}
  \label{ass:t}
  Let $\Cat{C}$ be a symmetric monoidal stable $\infty$--category 
  equipped with a $t$--structure satisfying the following assumptions 
  \cite[Construction  8.4.3.9]{HigherAlgebra}:
  \begin{enumerate}[(i)]
    \item The $\infty$--category $\Cat{C}$ is presentable.
    \item The tensor product $\otimes \colon \Cat{C} \times \Cat{C} \to 
      \Cat{C}$ preserves small colimits separately in each variable.
    \item The full subcategory $\Cat{C}_{\geq 0} \subseteq \Cat{C}$ 
      contains the unit object and is closed under tensor products.
  \end{enumerate}
\end{ass}

Denote by $\CAlg (\Cat{C})$ the category of commutative algebra objects 
in $\Cat{C}$. We will make constant use of several fundamental facts 
proven in \cite{HigherAlgebra}. The first concerns the existence of a 
cotangent complex in such a situation. Lurie proves that in this 
generality for every commutative algebra object $A \in \CAlg(\Cat{C)}$ 
there exists a cotangent complex $L_A$ which is an $A$--module \cite[Theorem  
8.3.4.18]{HigherAlgebra}. Note that in the case where $A$ is an 
$\IE_{\infty}$--ring, the homotopy groups of $L_A$ are the topological 
Andr\'e--Quillen homology groups of $A$, and in characteristic different 
from zero these do not have to coincide with classical Andr\'e--Quillen 
homology groups. The second concerns the question what the cotangent 
complex classifies.  It turns out that maps from the cotangent complex 
$L_A$ of an object $A \in \CAlg (\Cat{C})$ to an $A$--module $M[1]$ 
correspond to square-zero extensions of $A$ with fiber $M$. To make a 
precise statement we recall the following definitions from \cite[Section  
8.4]{HigherAlgebra}.

The $\infty$--category of \emph{derivations} in $\CAlg(\Cat{C})$ consists 
of pairs $(A, \eta \colon L_A \to M[1])$ where $L_A$ is the cotangent 
complex of $A$ and $M$ is an $A$--module. This category will be denoted 
by $\Der(\CAlg(\Cat{C}))$. We can now impose connectivity assumptions on 
$A$ and the module $M$. Let $\Der_{n-\con}(\CAlg(\Cat{C}))$ be the full 
subcategory of \emph{$n$--connective derivations}, defined by the 
conditions that $A \in \Cat{C}_{\geq 0}$ and $M \in \Cat{C}_{\geq n}$.  
Imposing even stricter conditions, let $\Der_{n-\sm}(\CAlg(\Cat{C}))$ be 
the full subcategory of \emph{$n$--small derivations} of 
$\Der_{n-\con}(\CAlg(\Cat{C}))$ spanned by those objects such that $M 
\in \Cat{C}_{\leq 2n}$.

To each derivation we can associate a \emph{square-zero extension}. This 
associates to a derivation
\[
  (A, \eta \colon L_A \to M[1])
\]
a morphism
\[
  A^{\eta} \longrightarrow A
\]
of objects in $\CAlg(\Cat{C})$. The fiber of $A^{\eta} \to A$ can be 
identified as an $A^{\eta}$--module with $M$. More generally, we say that 
a morphism $\widetilde{A} \to A$ is a square-zero extension if there 
exists a derivation $(A, L_A \to M[1])$ and an equivalence 
$\widetilde{A} \simeq A^{\eta}$. As above, we can impose connectivity 
assumptions on square-zero extensions. A morphism $f \colon A \to B$ in 
$\CAlg(\Cat{C})$ is an $n$--connective extension if $A \in \Cat{C}_{\geq 
  0}$ and $\fib (f) \in \Cat{C}_{\geq n}$. Again imposing further 
connectivity assumptions, we call an extension $n$--small if $\fib (f) 
\in \Cat{C}_{\leq 2n}$ and the multiplication map $\fib(f) \otimes_A 
\fib(f) \to \fib(f)$ is nullhomotopic. Denote by $\Fun _{n-\con} 
(\Delta^1, \CAlg(\Cat{C}))$ the full subcategory of the category of 
morphisms $\Fun (\Delta^1, \CAlg{C})$ spanned by the $n$--connective 
extensions, and by $\Fun _{n-\sm} (\Delta^1, \CAlg(\Cat{C}))$ the full 
subcategory spanned by the $n$--small extensions.

The process described above in fact defines a functor of 
$\infty$--categories
\[
  \Phi \colon \Der (\CAlg (\Cat{C})) \longrightarrow \Fun (\Delta^1, 
  \CAlg(\Cat{C}))
\]
given on objects by
\[
  (A, \eta \colon L_A \to M[1]) \longmapsto (A^{\eta} \to A).
\]
This functor has a left adjoint
\[
  \Psi \colon  \Fun (\Delta^1, \CAlg(\Cat{C})) \longrightarrow  \Der 
  (\CAlg (\Cat{C}))
\]
given on objects by
\[
  (\widetilde{A} \to A) \longmapsto (A, d \colon L_A \to 
  L_{A/\widetilde{A}}).
\]

Lurie proves that this adjunction restricts to subcategories with the 
appropriate connectivity assumptions and gives an equivalence of 
categories.

\begin{thm}\cite[Theorem 8.4.1.26]{HigherAlgebra}
  \label{thm:extension}
  Let $\Cat{C}$ be as above. Then
  \[
    \Phi_{n-\sm} \colon \Der_{n-\sm} \longrightarrow \Fun 
    _{n-\sm}(\Delta ^1, \CAlg(\Cat{C}))
  \]
 is an equivalence of $\infty$--categories.
\end{thm} 

The third fundamental fact we will use concerns the connectivity of the 
cotangent complex. In short, the cotangent complex of a highly connected 
morphism is again highly connected. The precise statement is the 
following:

\begin{thm}
  \cite[Theorem 8.4.3.11]{HigherAlgebra}
  \label{thm:connectivity}
  Let $\Cat{C}$ be an $\infty$--category as above. Let $f \colon A \to B$ 
  be a morphism of objects of $\CAlg(\Cat{C})$ such that both $A$ and $B 
  \in \Cat{C}_{\geq 0}$. Assume that $\cofib(f) \in \Cat{C}_{\geq n}$.  
  Then there exists a canonical morphism $\epsilon_f \colon B \otimes _A 
  \cofib (f) \to L_{B/A}$, and furthermore $\fib (\epsilon _f ) \in 
  \Cat{C}_{\geq 2n}$.
\end{thm}

In the special case of square-zero extension $f \colon A^{\eta} \to A$ 
with cofiber $M[1]$, the map $\epsilon_f$ allows us to compare $M[1]$ 
with the relative cotangent complex $L_{A/A^{\eta}}$.

\subsection{The Construction}

We begin by giving the definition of an obstruction theory in this 
abstract setting.

\begin{defn}
  \label{defn:pot}
  Let $A \in  \Cat{C}_{\geq 0}$ be a commutative algebra object.  An 
  \emph{$n$--connective obstruction theory} for $A$ is a morphism
  \[
    \phi \colon E \longrightarrow L_{A}
  \]
  of connective $A$--modules such that $\cofib(\phi) \in \Cat{C}_{\geq 
    n+1}$.
\end{defn}

\begin{rem}
  Let $A$ be a discrete object of $\CAlg(\Cat{C})$ equipped with a 
  1--connective obstruction theory $\phi \colon E \to L_A$. The 
  condition $\cofib{\phi} \in \Cat{C}_{\geq 2}$ is equivalent to $\pi_0 
  \phi$ being an isomorphism  and $\pi_1 \phi$ being surjective, thus 
  recovering the definition of \cite{behrend97}.
\end{rem}

\begin{rem}
  \label{rem:assder}
  The datum of an $n$--connective obstruction theory for connective $A 
  \in \CAlg(\Cat{C})$ is equivalent to giving an $n$--connective 
  derivation of $A$. To see this, simply complete $\phi \colon E \to 
  L_A$ to a cofiber sequence
  \[
    E \overset{\phi}{\to} L_A \overset{\eta}{\to} K.
  \]
  By definition, $\eta \colon L_A \to K$ is an $n$--connective derivation.
\end{rem}

\begin{defn}
  \label{defn:induce}
  Let $n \geq 1$ and let $(A, \phi \colon E \to L_A)$ be a 
  $n$--connective obstruction theory, and $(A, \eta \colon L_A \to K)$ 
  the associated $n$--connective derivation. We say that a pair
  \[
    \left( f \colon B \to A, \widetilde{\delta} \colon K \to 
      L_{A/B}\right)
  \]
  \emph{induces the obstruction theory} if
  \begin{enumerate}[(i)]
    \item $\trunc{n-1}{f} \colon \trunc{n-1}{B} \to \trunc{n-1}{A}$ is 
      an equivalence.
    \item The diagram
      \[
        \xymatrix{
          L_A \ar[r]^{\phi} \ar[d]_{d} & K \ar[dl]^{\widetilde{\delta}} 
          \\
          L_{A/B} &
        }
      \]
      commutes and $\widetilde{\delta}$ is an equivalence.
  \end{enumerate}
\end{defn}

\begin{rem}
  \label{rem:assphi}
  Let $(A, \phi \colon E \to L_A)$ be an $n$--connective obstruction 
  theory, and assume that $( f \colon B \to A, \widetilde{\delta} \colon 
  K \to L_{A/B} )$ induces the obstruction theory.  This induces an 
  equivalence
  \[
    \widetilde{\phi} \colon E \to A \otimes _B L_{B}.
  \]
\end{rem}

\begin{exmp}
  Let $A \in \CAlg(\Cat{C})$. Then $\pi_0 A$ can be equipped with a 
  canonical 1--connective obstruction theory. The obstruction theory is 
  given by
  \[
    \phi \colon \pi_0 A \otimes _A L_A \longrightarrow L_{\pi_0 A}.
  \]
  It immediately follows from \fullref{thm:connectivity} that the 
  cofiber of $\phi$ is in $\Cat{C}_{\geq 2}$. This obstruction theory is 
  trivially induced by $(A \to \pi_0 A, \id \colon L_{\pi _0 A/A} \to 
  L_{\pi_0 A/A})$. More generally, $A$ induces an $n+1$--connective 
  obstruction theory on all truncations $A \to \trunc {n}{A}$.  
\end{exmp}

Starting from the data of an $n$--connective obstruction theory $(A, \phi 
\colon E \to L_A)$, or equivalently, an $n$--connective derivation $(A, 
\eta \colon L_A \to K)$, we now want to examine when a pair  $(f \colon  
B \to A, \widetilde{\delta} \colon K \to L_{B/A})$ inducing the 
$n$--connective derivation exits. We will aim to construct $B$ as an 
increasingly connective tower of square-zero extensions of $A$ and 
encounter some obstruction on the way. We begin with a simple result on 
the connectivity of square-zero extensions.

\begin{lem}
  \label{lem:conn}
  Let $(A, \eta \colon L_A \to M[1])$ be an $n$--connective derivation.  
  Then the square-zero extension $A^{\eta} \to A$ is $n$--connective.
\end{lem}
\begin{proof}
  We have a fiber sequence of $A^{\eta}$--modules
  \[
    M \to A^{\eta} \to A.
  \]
  Since $M \in \Cat{C}_{\geq n}$ by assumption, the claim follows.
\end{proof}

We now introduce the key technical tool. We have seen that given a 
$n$--connective derivation $(A, \eta \colon L_A \to M[1])$ there exists 
an associated $n$--connective square-zero extension $f \colon A^{\eta} 
\to A$.  We now want to study how the relative cotangent complex 
$L_{A/A^{\eta}}$ compares to the module $M[1]$. In the following we will 
construct a morphism $\delta_f$ that compares the two. This 
$\delta_f$--map is a slight refinement of the map $\epsilon_f$ of 
\fullref{thm:connectivity}, and it can also be directly deduced from 
$\epsilon_f$.

Recall that we have an adjunction $\Phi \leftrightarrows \Psi$ between 
the categories of extensions and derivations. Let $v$ be the co-unit of 
this adjunction. By definition, on a derivation $(A, \eta \colon L_A \to 
M[1])$ with corresponding extension $f \colon A^{\eta} \to A$ the 
co-unit $v$ is given by
\[
  (A, \eta \colon L_A \to M[1]) \longmapsto (A, d \colon L_A \to 
  L_{A/A^{\eta}}).
\]
In particular, we obtain the following diagram in the category $\Mod_A$:
\begin{equation}
  \label{eq:delta}
  \xymatrix{
    L_{A} \ar[r]^{\eta} \ar[dr]_{d} & M[1] \ar[d]^{\delta_f} \\
    & L_{A/A^{\eta}}
  }
\end{equation}

\begin{defn}
  Let $(A, \eta \colon L_A \to M[1])$ be a derivation. Let $\delta_f$ be 
  the morphism defined by the co-unit $v$ of the adjunction $\Phi 
  \leftrightarrows \Psi$ as in Equation \eqref{eq:delta}.
\end{defn}

Given an object $A \in \CAlg (\Cat{C})$ with obstruction theory $\phi 
\colon E \to L_A$ and associated derivation $\eta \colon L_A \to K$ the 
morphism $\delta_f$ fits into the fundamental diagram of cofiber 
sequences:
\begin{equation}
  \label{eq:FundDiag}
  \xymatrix{
    E \ar[r]^{\phi'} \ar[d] & A \otimes _{A^{\eta}} L_{A^{\eta}} 
    \ar[r]^{\eta '} \ar[d] & \fib(\delta_f)  \ar[d] \\
    E \ar[r]^{\phi} \ar[d] & L_{A} \ar[d]^{d} \ar[r]^{\eta} & K 
    \ar[d]^{\delta_f} \\
    0 \ar[r] & L_{A/A^{\eta}} \ar[r] & L_{A/A^{\eta}}
  }
\end{equation}

We next prove a connectivity estimate for $\delta_f$ analogous to the 
connectivity estimate of \fullref{thm:connectivity} for $\epsilon_f$.  
This result again could also be easily deduced from the result for 
$\epsilon_f$.

\begin{prop}
  \label{prop:delta}
  Let $(A, \eta \colon L_A \to M[1])$ be an $n$--connective derivation, 
  and let $f \colon A^{\eta} \to A$ be the corresponding square-zero 
  extension.  Then $\fib (\delta_f) \in \Cat{C}_{\geq 2n+2}$, where 
  $\delta_f \colon M[1] \to L_{A/A^{\eta}}$ is the canonical morphism.
\end{prop}
\begin{proof}
  We have to show that $\trunc{2n+1}{\delta_f}$ is an equivalence. But 
  $\trunc{2n+1}{\delta_f}$ is the co-unit of the adjunction $ \Phi 
  _{n-\sm}\leftrightarrows \Psi_{n-\sm}$, which is an equivalence.  
\end{proof}

Thus given an $n$--connective obstruction theory $(A, \phi \colon E \to 
L_A)$ we can form the square-zero extension $A^{\eta} \to A$ 
corresponding to the cofiber sequence
\[
  E \overset{\phi}{\to} L_A \overset{\eta}{\to} K
\]
and the cotangent complex $A^{\eta}$ is an excellent approximation of 
$E$.  More precisely, the canonical morphism $\phi' \colon E \to A 
\otimes _{A^{\eta}} L_{A^{\eta}}$ is $(2n+3)$--connective.

To go further, we have to assume that the morphism $ \phi' \colon E \to 
A \otimes _{A^{\eta}} L_{A^{\eta}}$ lifts to the square-zero extension 
$A^{\eta} \to A$.

\begin{defn}
  Let $(A, \phi \colon E \to L_A)$ be an $n$--connective obstruction 
  theory, and let $A^{\eta} \to A$ be the corresponding square-zero 
  extension. The obstruction theory \emph{lifts to $A^{\eta}$} if there 
  exists an obstruction theory
  \[
    \phi^{\eta} \colon E^{\eta} \to L_{A^{\eta}}
  \]
  such that
  \[
    \xymatrix{
      A \otimes_{A^{\eta}} E^{\eta} \ar[r] \ar[d]_{\simeq} & A 
      \otimes_{A^{\eta}} L_{A^{\eta}} \ar[d]^{\simeq}\\
      E \ar[r]^{\phi'} & A \otimes_{A^{\eta}} L_{A^{\eta}}
    }
  \]
  commutes and the left vertical map is an equivalence.
\end{defn}
Assuming that the obstruction theory lifts to $A^{\eta}$, we can define 
a further square-zero extension $A^{\theta} \to A^{\eta}$. This 
square-zero extension is defined by completing the lifted obstruction 
theory to a cofiber square
\[
  E^{\eta} \overset{\phi^{\eta}}{\to} L_{A^{\eta}} \overset{\theta}{\to} 
  K^{\eta}
\]
and taking the square-zero extension corresponding to $\theta$. We want 
to know the connectivity properties of $K^{\eta}$.
\begin{lem}
  \label{lem:KConnect}
  Assume that $\phi \colon E \to L_A$ is an $n$--connective obstruction 
  theory that lifts to $A^{\eta}$. Then $K^{\eta}$ is 
  $(2n+2)$--connective.
\end{lem}
\begin{proof}
  By definition we have $A \otimes_{A^{\eta}} K^{\eta} \simeq 
  \fib(\delta_f)$, and $\fib(\delta_f)$ is $(2n+2)$--connective. Assume 
  there exists a $j < (2n+2)$ such that $\pi_j(K^{\eta})$ is not zero, 
  and assume that $j$ is minimal. Since $\pi_0 A^{\eta} = \pi_0 A$ 
  holds, we then have
  \begin{align*}
    \pi_j (K^{\eta}) &= \pi_j (K^{\eta}) \otimes _{\pi_0 A^{\eta}} \pi_0 
    A = \pi_j (K^{\eta} \otimes_{A^{\eta}} \pi_0 A)\\
    &= \pi_j (K^{\eta} \otimes_{A^{\eta}} \otimes A \otimes_A \pi_0 A) = 
    \pi_j ( \fib(\delta_f) \otimes_A \pi_0 A)
  \end{align*}
  non-zero, which contradicts that $\fib(\delta_f)$ is 
  $(2n+2)$--connective.
\end{proof}

We thus deduce that the derivation $\theta \colon L_{A^{\eta}} \to 
K^{\eta}$ is $(2n+1)$--connective.

We now want to study if the cotangent complex of $L_{A^{\theta}}$ gives 
a better approximation of $E$ than $L_{A^{\eta}}$.

\begin{lem}
  \label{lem:epsilon_lift}
  Assume that $\phi \colon E \to L_A$ lifts to $A^{\eta}$, and let $g 
  \colon A^{\theta}\to A^{\eta}$ be the corresponding square-zero 
  extension with fundamental diagram
  \[
    \xymatrix{
      E^{\eta} \ar[r]^{\phi^{\eta \prime}} \ar[d] & A^{\eta} \otimes 
      _{A^{\theta}} L_{A^{\theta}} \ar[r]^{\theta '} \ar[d] & 
      \fib(\delta_g)  \ar[d] \\
      E^{\eta} \ar[r]^{\phi^{\eta}} \ar[d] & L_{A^{\eta}} \ar[d]^{d} 
      \ar[r]^{\theta} & K^{\eta} \ar[d]^{\delta_g} \\
      0 \ar[r] & L_{A^{\eta}/A^{\theta}} \ar[r] & 
      L_{A^{\eta}/A^{\theta}}
    }
  \]
  Let $K = \cofib(\phi)$. Then  a canonical $(4n+4)$--connective 
  morphism
  \[
    \delta_{gf} \colon K \to L_{A/A^{\theta}}
  \]
  exists.
\end{lem}
\begin{proof}
  From the composition $A^{\theta} \overset{g}{\to} A^{\eta} 
  \overset{f}{\to} A$ we have the cofiber sequence $L_{A/A^{\theta}} \to 
  L_{A/A^{\eta}} \to A \otimes_{A^{\eta}} L_{A^{\eta}/A^{\theta}}[1]$.  
  Using $\delta_f$ and $\delta_g$ we obtain a diagram of cofiber 
  sequences
  \[
    \xymatrix{
      K \ar[r] \ar[d]_{\delta_{gf}} & K \ar[r] \ar[d]_{\delta_f} & 0 
      \ar[d]\\
      L_{A/A^{\theta}} \ar[r] \ar[d] & L_{A/A^{\eta}} \ar[r] \ar[d] & A 
      \otimes_{A^{\eta}}L_{A^{\eta}/A^{\theta}}[1] \ar[d]^{=}\\
      \fib( \id \otimes \delta_g [1]) \ar[r] & A \otimes_{A^{\eta}} 
      K^{\eta}[1] \ar[r]^{\id \otimes \delta_g[1]}&A 
      \otimes_{A^{\eta}}L_{A^{\eta}/A^{\theta}}[1]\\
    }
  \]
  defining $\delta_{gf}$.

  Since $\theta \colon L_{A^{\eta}} \to K^{\eta}$ is a 
  $(2n+1)$--connective derivation, it follows that $\delta_g$ is 
  $(4n+4)$-connective, and thus $\id \otimes \delta_g[1]$ is 
  $(4n+5)$--connective. Identifying $\cofib(\delta_{gf})$ with $\fib(\id 
  \otimes \delta_g[1])$ the connectivity claim follows.
\end{proof}
\begin{cor}
  Assume that $\phi \colon E \to L_A$ lifts to $A^{\eta}$, and let $g 
  \colon A^{\theta}\to A^{\eta}$ be the corresponding square-zero 
  extension. Then there is a canonical $(4n+3)$--connective morphism
  \[
    E \to A \otimes _{A^{\theta}} L_{A^{\theta}}
  \]
\end{cor}
\begin{proof}
  Using $\delta_{gf}$, we obtain a diagram of cofiber squares
  \[
    \xymatrix{
      E \ar[r] \ar[d] & A \otimes_{A^{\theta}} L_{A^{\theta}} \ar[r] 
      \ar[d] & \fib(\delta_{gf}) \ar[d]\\
      E \ar[r] \ar[d] & L_A \ar[r] \ar[d] & K \ar[d]^{\delta_{gf}}\\
      0 \ar[r] & L_{A/A^{\theta}} \ar[r] & L_{A/A^{\theta}}
    }
  \]
\end{proof}

Finally we need a result allowing us to compute the cotangent complex of 
an increasingly connected tower of square-zero extensions.

\begin{lem}
  \label{lem:cotoftower}
  Let
  \[
    A_0 \overset{f_1}{\longleftarrow} A_1 \overset{f_2}{\longleftarrow} 
    A_2 \overset{f_3}{\longleftarrow} \dots
  \]
  be a sequence of square-zero extensions where $f_n$ is $n$--connective.  
  Let $B$ be the inverse limit $\lim \{A_n\}$. Then
  \[
    L_B \simeq \lim \{ L_{A_n} \}.
  \]
\end{lem}
\begin{proof}
  Passing to the Postnikov decomposition of $B$ we have a sequence of 
  equivalences:
  \[
    \xymatrix{
      \trunc{0}{B} \ar[d]_{\simeq} & \trunc{1}{B} \ar[d]_{\simeq} \ar[l] 
      &\trunc{2}{B} \ar[d]_{\simeq} \ar[l]& \ar[l] \dots \\
      \trunc{0}{A_0} & \trunc{1}{A_1} \ar[l] & \trunc{2}{A_2} \ar[l] & 
      \dots \ar[l]
    }
  \]
  This induces equivalences on the Postnikov decomposition of $L_B$
  \[
    \xymatrix{
      \trunc{0}{L_B} \ar[d]_{\simeq} & \trunc{1}{L_B} \ar[d]_{\simeq} 
      \ar[l] &\trunc{2}{L_B} \ar[d]_{\simeq} \ar[l]& \ar[l] \dots \\
      \trunc{0}{L_{A_0}} & \trunc{1}{L_{A_1}} \ar[l] & 
      \trunc{2}{L_{A_2}} \ar[l] & \dots \ar[l].
    }
  \]
  and the claim follows.
\end{proof} 

To phrase the main result we have to introduce a bit of notation. Given 
an $n$--connective obstruction theory $\phi \colon E \to L_A$ with 
associated derivation $ \eta \colon L_A \to K$, we set $A_1 := A^{\eta}$ 
to be corresponding square-zero extension. Assuming that it is possible 
to choose a lift of the obstruction theory to $A_1$, we can define a 
further square-zero extension $A_2 \to A_1$. If we again assume that it 
is possible to choose a lift of the obstruction theory to $A_2$, we 
obtain a further square-zero extension $A_3$ and so on. We fix this in 
the following definition.

\begin{defn}
  Let $(A_0, \phi_0 \colon E_0 \to L_{A_0})$ be an $n$--connective 
  obstruction theory, and let $A_1 \to A_0$ be the corresponding 
  square-zero extension. An \emph{inductive system of lifts} of the 
  obstruction theory exists if for all $m \geq 0$ there exists a lift 
  $\phi_{m+1} \colon E_{m+1} \to L_{A_{m+1}}$ of the obstruction theory 
  $\phi_{m}$ to $A_{m+1}$. Here $A_{m+1} \to A_m$ is the square-zero 
  extension defined by the obstruction theory $\phi_m$.
\end{defn}

Note that it could happen that a certain choice of a lift $\phi_1 \colon 
E_1 \to L_{A_1}$ does not lift to the subsequently defined square-zero 
extension $A_2 \to A_1$, but a different choice $\phi_1' \colon E'_1 \to 
L_{A_1}$ does lift to the square-zero extension $A'_2 \to A_1$.

We now have all tools to prove the main result. Assuming that an 
inductive system of lifts of the obstruction theory exists, we will use 
the previous results to give a tower of increasingly connected 
square-zero extensions $A_{m+1}\to A_m$.  In every step we will measure 
the difference between $L_{A/A_{m}}$ and $K$ using the maps defined in 
\fullref{lem:epsilon_lift}. In every step the degrees in which a defect 
still exists will be pushed up by a factor of 2.

\begin{thm}
  \label{thm:algresult}
  Let $\Cat{C}$ be an $\infty$--category as in \fullref{ass:t}, and let 
  $A \in \CAlg(\Cat{C})$ be a connective commutative algebra object.  
  Assume that $(A, \phi \colon E \to L_A)$ is an $n$--connective 
  obstruction theory with $n \geq 1$, and let $\cofib(\phi)=K$.
  
  Then a pair
  \[
    \left( f \colon B \to A, \widetilde{\delta} \colon K \to L_{A/B} 
    \right)
  \]
  inducing the obstruction theory exists if and only if an inductive 
  system of lifts of the obstruction theory exists.
\end{thm}
\begin{proof}
  We first assume that the an inductive system of lifts of the 
  obstruction theory exists.  Let $A = A_0$, and let $\eta_0 \colon 
  L_{A_0} \to K$ be the $n$--connective derivation associated to the 
  obstruction theory.  Since an inductive system of lifts of the 
  obstruction theory exists, we can inductively define a tower of 
  increasingly connected square-zero extensions $f_{m+1,m} \colon 
  A_{m+1} \to A_m$, where by \fullref{lem:KConnect} $f_{m+1,m}$ is $2^m 
  (n+1)-1$--connective.

  Now denote by $f_m$ the composition $f_{1,0} \circ \dots \circ 
  f_{m,m-1}$. Set $\delta_m$ to be the $2^{m+1}(n+1)$--connective 
  morphism obtain by applying \fullref{lem:epsilon_lift} to the 
  morphisms
  \[
    A_{m+1} \overset{f_{m+1,m}}{\longrightarrow} A_m 
    \overset{f_m}{\longrightarrow} A_0.
  \]
  Now define $B$ to be the inverse limit $\lim \{ A_m \}$. Using the 
  maps $\delta_m$ we have a series of maps
  \[
    \xymatrix{
       & & & K \ar[d]^{\delta_1} \ar[dl]^{\delta_2} \ar[dll]^{\delta _3} 
       \ar[dlll]\\
       \dots\ar[r]  & L_{A_0/A_3} \ar[r] & L_{A_0/A_2}\ar[r]& 
       L_{A_0/A_1}
    }
  \]
  where $\delta_m$ is $2^m(n+1)$--connective. Passing to the limit and 
  using \fullref{lem:cotoftower}, we have an equivalence 
  $\widetilde{\delta} \colon K \to \lim{L_{A/A_m}} \simeq L_{A/B}$.

  Conversely, assume that thre exists a pair $\left( f \colon B \to A, 
    \widetilde{\delta} \colon K \to L_{A/B} \right)$ inducing the 
  obstruction theory. We begin by proving that the induced obstruction 
  theory lifts to $A_1$. Recall that $A_1 \to A$ is defined to be the 
  square-zero extension corresponding to the derivation $\eta$ in the 
  cofiber sequence $A \otimes_B L_{B} \to L_A \overset{\eta}{\to} 
  L_{A/B}$. Thus the obstruction to lifting $B \to A$ to $A_1$ vanishes, 
  and we obtain a diagram
  \[
    \xymatrix{
      &A_1\ar[d]\\
      B\ar[ur] \ar[r] & A.
    }
  \]
  Now the obstruction theory obviously lifts to $A_1$ using $A_1 \otimes 
  _B L_{B} \to L_{A_1}$. Finally, assume the obstruction theory lifts to 
  the $m$--th inductively defined square-zero extension. We have to 
  prove that it lifts to the $(m+1)$--st level. By definition, we have a 
  tower of morphisms
  \[
    \xymatrix{
       & & & B \ar[d] \ar[dl] \ar[dll] \\
       A_{m+1} \ar[r]  & A_m \ar[r] & \dots \ar[r]& A
    }
  \]
  Again using the same argument as above, the obstruction to lifting $B 
  \to A_m$ to $A_{m+1}$ vanishes, and the obstruction theory lifts by 
  using $A_{m+1} \otimes_ B L_{B} \to L_{A_{m+1}}$.
\end{proof}

We can now apply this result in some examples.

\begin{exmp}
  Let $\Cat{C}=\Sp$ be the $\infty$--category of spectra. An object of 
  $\CAlg (\Sp)$ then is an $\IE_{\infty}$--ring. Discrete objects of 
  $\CAlg (\Sp)$ can be identified with ordinary commutative rings.  
  Applying the above theorem, it follows that a 1--connective 
  obstruction theory for a commutative ring $A$ is induced by some 
  $\IE_{\infty}$--ring $B$ with $\pi_0 B =A$ if and only if an inductive 
  system of lifts of the obstruction theory exists.
\end{exmp}

\begin{exmp}
  \label{exmp:overk}
  Let $k$ be a connective $\IE_{\infty}$--ring and let $\Cat{C}$ denote 
  the category $\Mod_{k}(\Sp)$ of $k$--module spectra. Define $\CAlg_k$ 
  to be $\CAlg(\Mod_k (\Sp))$. A discrete object of $\CAlg_k$ is an 
  ordinary commutative algebra over $\pi_0 k$. By the above theorem a 
  1--connective obstruction theory for a commutative $\pi_0 k$--algebra 
  $A$ is induced by an $\IE_{\infty}$--algebra $B$ over $k$ such that 
  $\pi_0 B=A$ if and only if an inductive system of lifts of the 
  obstruction theory exists. \footnote{The category $\CAlg_k$ can be 
    identified with the under-category $\CAlg(\Sp)_{k/}$. Under this 
    identification the cotangent complex associated to the category 
    $\CAlg_k$ can be identified with the relative cotangent complex 
    $L_{A/k}$.}
\end{exmp}

\begin{exmp}
  \label{exmp:simplicial}
  In the previous example, let $k$ be an ordinary commutative ring 
  containing the rationals $\IQ$ viewed as a discrete 
  $\IE_{\infty}$--ring. Then we can identify connective objects of 
  $\CAlg_k$ with the nerve of the category of simplicial commutative 
  $k$--algebras \cite[Proposition 8.1.4.20]{HigherAlgebra}. Applying the 
  above theorem, a 1--connective obstruction theory for a commutative 
  $k$--algebra is induced by a simplicial commutative $k$--algebra if 
  and only if an inductive system of lifts of the obstruction theory 
  exists.
\end{exmp}

\begin{exmp}
  \label{exmp:cdga}
  Taking $k$ as in the previous example, a further explicit model for 
  the $\infty$--category $\CAlg_k$ is given by commutative differential 
  graded algebras over $k$ \cite[Proposition 8.1.4.11]{HigherAlgebra}.  
  The above theorem thus shows that a 1--connective obstruction theory 
  on a commutative $k$--algebra is induced by a connective commutative 
  differential graded algebra if and only if an inductive system of 
  lifts of the obstruction theory exists.
\end{exmp}

\begin{rem}
  An important case not covered by \fullref{thm:algresult} is simplicial 
  algebras over $k$ where $k$ is a ordinary commutative ring not 
  necessarily containing $\IQ$. This case is important since it is a 
  homotopical algebra context in the sense of \cite{toen08} and leads to 
  derived algebraic geometry over any base ring. This case is not 
  covered by \fullref{thm:algresult} since simplicial algebras over $k$ 
  no longer provide a model for $\IE_{\infty}$--algebras over $k$ if $k$ 
  does not contain $\IQ$. Thus it is not directly possible to apply the 
  formalism developed in \cite{HigherAlgebra} to this case.

  Nevertheless, it should be possible to extend all results to this 
  case. The main tool we have used is comparing the target module $M[1]$ 
  in a derivation $(A, \eta \colon L_A \to M[1])$ to the relative 
  cotangent complex $L_{A/A^{\eta}}$ of the corresponding square-zero 
  extension $f \colon A^{\eta} \to A$ via the map $\delta_f$. This map 
  was defined via the adjunction between derivations and square-zero 
  extensions, and its connectivity properties where deduced from 
  \fullref{thm:extension}.  Proving the analogous results in the context 
  of simplicial algebras over any base $k$ would provide all necessary 
  tools to carry out the proof. Several results in this direction can be 
  found in \cite{toen08}. There the relative cotangent complex 
  $L_{A/A^{\eta}}$ is explicitly computed (\cite[Lemma  
  1.4.3.7]{toen08}) and is shown to have the same homotopy groups as 
  $M[1]$ up to degree $n+1$ for an $n$--connective module $M$ 
  (\cite[Lemma 2.2.2.7]{toen08}).  The main difficulty lies in extending 
  the connectivity estimate to the estimate of \fullref{prop:delta}.
\end{rem}

In \fullref{exmp:overk} one will typically impose some finiteness 
property on both the $\pi_0 k$--algebra $A$ and the obstruction theory 
$\phi \colon E \to L_{A /k}$.  We recall the relevant definition.

\begin{defn}
  Let $k$ be a connective $\IE_{\infty}$--ring, and let $A$ be a finitely 
  presented discrete commutative $\pi_0 k$--algebra equipped with a 
  1--connective obstruction theory $\phi \colon E \to L_A$. Let $m \geq 
  0$.
  \begin{enumerate}[(i)]
    \item The obstruction theory is an $m$--obstruction theory if $E \in 
      (\Mod_A)_{\leq m}$.
    \item The obstruction theory is perfect if $E$ is a perfect 
      $A$--module.
  \end{enumerate}
\end{defn}

We next want to ensure that if we start with a finitely presented 
commutative $\pi_0 k$--algebra equipped with a 1--connective perfect 
$m$--obstruction theory, any $k$--algebra inducing the obstruction theory 
satisfies a strong finiteness property. We briefly recall the relevant 
finiteness property following \cite[Definition  
8.2.5.26]{HigherAlgebra}.

\begin{defn}
  Let $k$ be a connective $\IE_{\infty}$--ring, and let $\CAlg_k=\CAlg 
  (\Mod_k(\Sp))$. A commutative $k$--algebra $B$ is \emph{locally of 
    finite presentation over $k$} if $B$ is a compact object of 
  $\CAlg_k$.
\end{defn}

Note that a finitely presented discrete commutative $\pi_0 k$--algebra 
viewed as an object in $\CAlg_k$ will usually not satisfy the above 
finiteness property. As an example one can take any finitely presented 
discrete commutative $\pi_0 k$--algebra with non-perfect cotangent 
complex, as in light of \cite[Theorem  8.4.3.17]{HigherAlgebra} perfectness 
of the cotangent complex is necessary for being locally of finite 
presentation in $\CAlg_k$.  Nevertheless, we will later see that given a 
finitely presented discrete commutative $\pi_0 k$--algebra $A$ equipped 
with an $n$--connective perfect $m$--obstruction theory, any commutative 
$k$--algebra $B$ inducing the obstruction theory does satisfy the above 
finiteness property, although $A$ itself will usually not.

\begin{rem}
  Note that in \cite[Definition 1.2.3.1]{toen08} slightly different 
  terminology is used. There a compact object of $\CAlg_k$ is called 
  \emph{finitely presented}, whereas Lurie reserves finitely presented 
  for algebras which lie in the smallest full subcategory which contains 
  finitely generated free algebras and is stable under retracts and 
  finite colimits.
\end{rem}
  
Before we begin we need a series of lemmas. These lemmas will allow us 
to deduce properties of an $A$--module $M$ from the respective properties 
of the $\pi _0 A$--module $\pi_0 A \otimes _A M$. Recall the following 
definitions for a connective $\IE_{\infty}$--ring $A$ and an $A$--module 
$M$:
\begin{enumerate}[(i)]
  \item $M$ is of Tor--amplitude $\leq n$ if for any discrete $A$--module 
    $N$ we have
    \[
      \pi_i (N \otimes _A M)=0
    \]
    for $i > n$.
  \item $M$ is perfect to order $n$ if $\trunc{n}{M}$ is a compact 
    object of $(\Mod_A)_{\leq n}$.
  \item $M$ is almost perfect if it is perfect to order $n$ for all 
    integers $n$.
\end{enumerate}

\begin{lem}
  \label{lem:toramp}
  Let $A$ be a connective $\IE _{\infty}$--ring and $M$ a connective 
  $A$--module. If $\pi_0  A \otimes _{A}M  $ has Tor--amplitude $\leq n$ 
  as $\pi_0 A$--module, then $M$ has Tor--amplitude $\leq n$ as 
  $A$--module.
\end{lem}
\begin{proof}
  Let $N$ be a discrete $A$--module. In particular, $N$ is a 
  $\pi_0 A$--module. The claim then follows from
  \[
    \pi_i(N \otimes _A M) = \pi_i (N \otimes _{\pi_0 A} \pi_0 A \otimes 
    _A M) = \Tor _i ^{\pi_0 A} (N, \pi_0 A  \otimes_A M). \proved
  \]
\end{proof}

\begin{lem}
  \label{lem:almperf}
  Let $A$ be a connective $\IE _{\infty}$--ring, $M$ a connective 
  $A$--module, and $n \geq 0$. If $\pi_0  A \otimes _{A}M  $ is perfect 
  to order $n$ as $\pi_0 A$--module, then $M$ is perfect to order $n$ as 
  $A$--module.
\end{lem}
\begin{proof}
  We prove the claim by induction over $n$. For the case of $n=0$ recall 
  that $M$ is perfect to order 0 as $A$--module if and only if $\pi_0 M$ 
  is finitely generated as a module over $\pi_0 A$. The claim then 
  follows from
  \[
    \pi_0 M=\pi_0 A \otimes_{\pi_0 A} \pi_0 M = \pi_0 (\pi_0 A \otimes 
    _A M).
  \]
  Now let $n > 0$. Recall that given a map of $A$--modules $\phi \colon 
  A^k \to M$ which induces a surjection $\pi_0 A ^k \to \pi_0 M$, then 
  $M$ is perfect to order $n$ if and only if $\fib (\phi)$ is perfect to 
  order $n-1$ \cite[Proposition 2.6.12]{DAGVIII}. The argument now 
  follows \cite[Proposition 2.6.13]{DAGVIII}. As $\pi_0 M$ is finitely 
  generated, we can choose a fiber sequence of connective $A$--modules
  \[
    M' \longrightarrow A^k \longrightarrow M.
  \]
  Tensoring with $\pi_0 A$, we obtain a fiber sequence of $\pi_0 
  A$--modules
  \[
    \pi_0 A \otimes _A M' \longrightarrow \pi_0 A ^k \longrightarrow 
    \pi_0 A \otimes _A M.
  \]
  By assumption, $\pi_0 A \otimes_A M$ is perfect to order $n$ as $\pi_0 
  A$--module, so $\pi_0 A \otimes _A M'$ is perfect to order $n-1$ as 
  $\pi_0 A$--module.  By the inductive hypothesis $M'$ is perfect to 
  order $n-1$ as $A$--module, and thus $M$ is perfect to order $n$ as 
  $A$--module.
\end{proof}

\begin{cor}
  Let $A$ be a connective $\IE _{\infty}$--ring and $M$ a connective 
  $A$--module. If $ \pi_0 A \otimes _{A} M $ is almost perfect and of 
  finite Tor--amplitude as $\pi_0 A$--module, then $M$ is perfect as 
  $A$--module.
\end{cor}
\begin{proof}
  By \fullref{lem:toramp}, $M$ is of finite Tor--amplitude. By 
  \fullref{lem:almperf}, $M$ is perfect to order $n$ for all $n$ and 
  thus almost perfect. Now $M$ being almost perfect and of finite 
  Tor--amplitude imply that $M$ is perfect.
\end{proof}

We can now prove the finiteness result alluded to above.

\begin{prop}
  \label{prop:finpres}
  Let $k$ be a connective $\IE_{\infty}$--ring, and let $\CAlg_k=\CAlg 
  (\Mod_k(\Sp))$. Let $A$ be a discrete object of $\CAlg_k$ such that 
  $A$ is finitely presented as $\pi_0 k$--algebra. Assume that $A$ is 
  equipped with an $n$--connective perfect $m$--obstruction theory $\phi 
  \colon E \to L_{A/k}$. Then in any pair $(f \colon B \to 
  A,\widetilde{\delta} \colon L_{A/B} \to K)$ inducing the obstruction 
  theory, the object $B$ of $\CAlg_k$ is locally of finite presentation 
  and $L_{B/k}$ is of Tor--amplitude $\leq m$.
\end{prop}
\begin{proof}
  Using the equivalence $\widetilde{\phi}\colon E \to A \otimes _B 
  L_{B/k} = \pi_0 B \otimes _B L_{B/k}$ obtained from 
  $\widetilde{\delta}$ (see \fullref{rem:assphi}) it follows that $\pi_0 
  B \otimes _B L_{B/k}$ is perfect and of Tor--amplitude $\leq m$.  In 
  particular, $\pi_0 B \otimes _B L_{B/k}$ is almost perfect and of 
  finite Tor--amplitude $\leq m$. By the previous corollary, $L_{A/k}$ 
  is perfect and of Tor--amplitude $\leq m$. Now $\pi_0 B $ being of 
  finite presentation over $\pi_0 k$ and $L_{B/k}$ being perfect imply 
  that $B$ is locally of finite presentation over $k$ \cite[Theorem  
  8.4.3.17]{HigherAlgebra}.
\end{proof}

With certain finiteness conditions imposed, it is possible to define 
obstruction classes that measure if an obstruction theory lifts to the  
inductively defined square-zero extensions. For simplicity, we restrict 
to the situation described in \fullref{exmp:overk}. Here we have 
classifying stack for perfect complexes available (see \cite{tvaq, 
  pridham}) such that map induced on the tangent complexes by the 
classifying map is given by the Atiyah class.

Applying the formalism of Atiyah classes as in \cite{ht}, we can show 
that the complex defining an $n$--connective $n$--perfect obstruction 
theory always lifts to the corresponding square-zero extension.

\begin{lem}
  Assume we are in the situation of \fullref{exmp:overk}, and let $A \in 
  \CAlg(\Cat{C})$.  Let $\phi \colon E \to L_A$ be an $n$--connective 
  $n$--perfect obstruction theory. Let $\eta \colon L_A \to 
  \cofib(\phi)$ be the corresponding derivation and $A^{\eta} \to A$ the 
  corresponding square-zero extension. Then $E$ lifts to $A^{\eta}$.
\end{lem}
\begin{proof}
  The obstruction to lifting $E$ is given by the composition
  \[
    (\End(E)[1])^{\vee} \to L_A \to K
  \]
  Now since $K$ is $(n+1)$--connective and $(\End(E)[1])^{\vee}$ is 
  $(n-1)$--truncated, this morphism has to be homotopical to zero by the 
  $t$--structure.
\end{proof}

Under the above assumptions, denote the lifted complex by $E^{\eta}$.  
To lift an obstruction theory, the remaining problem thus is to lift the 
morphism $\phi' \colon E \to A \otimes_{A^{\eta}} L_{A^{\eta}}$ of 
\eqref{eq:FundDiag} to a morphism $\phi^{\eta} \colon E^{\eta} \to 
L_{A^{\eta}}$.  Whether or not this is possible can be measured by a 
certain obstruction class.

\begin{lem}
  Let $\phi \colon E \to L_A$ be an $n$--connective $n$--perfect 
  obstruction theory. Let $\eta \colon L_A \to K=\cofib(\phi)$ be the 
  corresponding derivation and $A^{\eta} \to A$ the corresponding 
  square-zero extension. Let $E^{\eta}$ be a perfect complex on 
  $A^{\eta}$ such that $A \otimes_{A^{\eta}} E^{\eta} \simeq E$. Then 
  $\phi'$ lifts to a morphism $\phi^{\eta} \colon E^{\eta} \to 
  L_{A^{\eta}}$ if and only if the class of the morphism
  \[
    E^{\eta} \to K \otimes_{A^{\eta}}  L_{A^{\eta}}
  \]
  in $\Ext^0_{A^{\eta}}(E^{\eta},K \otimes_{A^{\eta}}  L_{A^{\eta}})$ 
  defined in the proof vanishes.
\end{lem}
\begin{proof}
  Since $f \colon A^{\eta} \to A$ is a square-zero extension defined by 
  $\eta \colon L_A \to K$, we have $\cofib(f)=K$. Tensoring the 
  resulting cofiber sequence $A^{\eta} \to A \to K$ with $E^{\eta}$ and 
  $L_{A^{\eta}}$ respectively, we obtain a diagram of cofiber sequences
  \[
    \xymatrix{
      E^{\eta} \ar[r] & A \otimes_{A^{\eta}} E^{\eta} \ar[r] 
      \ar[d]_{\phi'} & K \otimes_{A^{\eta}} E^{\eta} \\
       L_{A^{\eta}}\ar[r] & A \otimes_{A^{\eta}}  L_{A^{\eta}}\ar[r] & K 
       \otimes_{A^{\eta}}  L_{A^{\eta}}\\
    }
  \]
  Thus $\phi'$ lifts if and only if the class of the composition
  \[
    E^{\eta} \to A \otimes_{A^{\eta}} E^{\eta} \to A \otimes_{A^{\eta}}  
    L_{A^{\eta}} \to K \otimes_{A^{\eta}}  L_{A^{\eta}}
  \]
  in $\Ext_{A^{\eta}}^{0}(E^{\eta}, K \otimes_{A^{\eta}}  L_{A^{\eta}})$ 
  vanishes.
\end{proof}

By combining the previous lemma with \fullref{thm:algresult}, we obtain 
the following Corollary.

\begin{cor}
  \label{cor:Atiyah}
  Let $\Cat{C}$ be an $\infty$--category as in \fullref{exmp:overk}, and 
  let $A \in \CAlg(\Cat{C})$.  Assume that $(A, \phi \colon E \to L_A)$ 
  is an $n$--connective $n$--perfect obstruction theory with $n \geq 1$, 
  and let $\cofib(\phi)=K$.
  
  Then a pair
  \[
    \left( f \colon B \to A, \widetilde{\delta} \colon K \to L_{A/B} 
    \right)
  \]
  inducing the obstruction theory exists if and only if the inductively 
  defined obstruction classes
  in $\Ext^0_{A_n}(E_n,K_{n-1} \otimes_{A_n}  L_{A_n})$
  vanish.
\end{cor}

\section{The Geometric Case}

In this section we want to apply the above results in the setting of 
spectral Deligne--Mumford stacks. We first recall some of the definitions 
we will use.

\subsection{Background}

In \cite{DAGVII} Lurie defines a $\infty$--category $\Sch 
(\Sh{G}^{\Sp}_{\et})$ of connective spectral Deligne--Mumford stacks. An 
object $\dS{X}$ of this category is a pair $\iT{X}$ consisting of an 
$\infty$--topos $\Sh{X}$ and a sheaf of $\IE_{\infty}$--rings satisfying 
further conditions. There also is a relative version $\Sch 
(\Sh{G}^{\Sp}_{\et}(k))$ of connective spectral Deligne--Mumford stacks 
over a connective $\IE_{\infty}$--ring $k$. Here the structure sheaf 
takes its values in the category $\CAlg_k$ of $\IE_{\infty}$--rings over 
$k$. This category can in fact be identified with the category of 
$\Sh{G}^{\Sp}_{\et}$--schemes equipped with a morphism to $\Spec 
(k)$.\footnote{In Lurie's notation, $\Spec k$ would be $\Spec^{\et} 
  (k)$. As we will only encounter this $\Spec$-functor omitting the 
  superscript hopefully does not lead to confusion.}

A key property of this category is that ordinary Deligne--Mumford stacks 
over the discrete commutative ring $\pi_0 k$ sit inside $\Sch 
(\Sh{G}^{\Sp}_{\et}(k))$ as the full subcategory spanned by the 
0--truncated and 1--localic $\Sh{G}^{\Sp}_{\et}(k)$--schemes.

The theory of spectral connective Deligne--Mumford stacks over $k$ is 
compatible with $n$--truncations in the sense that for every such stack 
$\dS{X}$ its $n$--truncation $\trunc{n}{\dS{X}}=(\Sh{X}, 
\trunc{n}{\Oof{X}})$ is again a spectral connective Deligne--Mumford 
stack over $k$. In particular, given a 1--localic connective spectral 
Deligne--Mumford stack, its 0--truncation 
$\trunc{0}{\dS{X}}=(\Sh{X},\pi_0 \Oof{X})$ is an ordinary 
Deligne--Mumford stack over $\pi_0 k$ with the same underlying 
$\infty$--topos as $\dS{X}$.  Furthermore, we have a canonical morphism 
$\trunc{0}{\dS{X}} \to \dS{X}$.

Finally, recall that given a connective spectral Deligne--Mumford stack 
$\dS{X}$ over $\Spec k$ we say that $\dS{X}$ is \emph{locally of finite 
  presentation over $\Spec k$} if it is possible to choose a covering by 
affine schemes $\Spec (A_{\alpha})$ such that each $A_{\alpha}$ is 
locally of finite presentation over $k$, i.e., a compact object of 
$\CAlg_k$. If a connective spectral Deligne--Mumford stack $\dS{X}$ is 
locally of finite presentation over $k$ and has a cotangent complex of 
Tor--amplitude $\leq 1$ we say that $\dS{X}$ is \emph{quasi-smooth}.

In the following we will make use of the following identification. Let 
$\Sh{X}$ be an $\infty$--topos, and take $\Cat{C}$ to be the 
$\infty$--category of sheaves of spectra $\Shv _{\Sp} (\Sh{X})$ on 
$\Cat{X}$.  This is a symmetric monoidal $\infty$--category equipped 
with a $t$--structure satisfying \fullref{ass:t}. Using the equivalence
\begin{equation}
  \label{eq:calg-object}
  \Shv_{\CAlg} (\Sh{X}) \simeq \CAlg(\Shv_{\Sp}(\Sh{X}))
\end{equation}
we can identify the structure sheaf $\Oof{X}$ of a connective spectral 
Deligne--Mumford stack with a commutative algebra object in 
$\Shv_{\Sp}(\Sh{X})$.

\subsection{The Construction}

We first give the definition of an $n$--connective obstruction theory in 
the geometric setting. The only difference to the algebraic case is that 
we want to assume the module defining the obstruction theory to be  
quasi-coherent.

\begin{defn}
  \label{defn:qcpot}
  Let $\dS{X}=\iT{X}$ be connective spectral Deligne--Mumford stack, and 
  let $n \geq 1$.  A $n$--connective quasi-coherent obstruction theory 
  for $\dS{X}$ is a morphism
  \[
    \phi \colon \Sh{E} \longrightarrow L_{\Oof{X}}
  \]
  of connective quasi-coherent $\Oof{X}$--modules such that $\cofib{\phi} 
  \in \QCoh(\dS{X})_{\geq n+1}$.
\end{defn}

We also have the analogous definition in the relative setting over 
$\Spec (k)$ using the relative cotangent complex.

\begin{defn}
  Let $k$ be a connective $\IE_{\infty}$--ring, and let $\dS{X}$ be a 
  connective spectral Deligne--Mumford stack over $\Spec (k)$. Let $n 
  \geq 1$.
  \begin{enumerate}[(i)]
    \item A $n$--connective quasi-coherent obstruction theory for 
      $\dS{X}$ over $\Spec (k)$ is a morphism
      \[
        \phi \colon \Sh{E} \longrightarrow L_{\Oof{X}/k}
      \]
      of quasi-coherent $\Oof{X}$--modules such that  $\cofib{\phi} \in 
      \QCoh(\dS{X})_{\geq n+1}$.
    \item Let $\dS{X}$ be locally of finite presentation over $\Spec 
      (k)$, and let $\phi \colon \Sh{E} \to L_{\Oof{X}/k}$ be a 
      $n$--connective obstruction theory. We say that the obstruction 
      theory is \emph{perfect} if $\Sh{E}$ is perfect. The obstruction 
      theory is an $m$--obstruction theory if $\Sh{E} \in \QCoh 
      (\dS{X})_{\geq 0} \cap \QCoh (\dS{X})_{\leq m}$.
  \end{enumerate}
\end{defn}

In complete analogy to \fullref{defn:induce} we have the notion of an 
object inducing the obstruction theory.

\begin{defn}
  Let $\dS{X}_0=(\Sh{X},\Sh{O}_{\Sh{X}_0})$ be a spectral connective 
  Deligne--Mumford stack equipped with an $n$--connective quasi-coherent 
  obstruction theory $\phi \colon \Sh{E} \to L_{\Sh{O}_{\Sh{X}_0}}$ with 
  $n \geq 1$. Let $\cofib(\phi)=\Sh{K}$, and let $\delta \colon \Sh{K} 
  \to L_{\Sh{O}_{\Sh{X}_0}}$ be the induced morphism. We say that the 
  pair
  \[
    \left(i \colon \dS{X}_0 \to \dS{X}, \widetilde{\delta} \colon \Sh{K} 
      \to L_{\dS{X}_{0}/\dS{X}}\right)
  \]
  induces the obstruction theory if
  \begin{enumerate}[(i)]
    \item $\dS{X}=\iT{X}$ is a connective spectral Deligne--Mumford stack 
      with the same underlying $\infty$--topos as $\dS{X}_0$.
    \item $\trunc{n}{\dS{X}}= \trunc{n}{\dS{X}_0}$.
    \item $\widetilde{\delta} \colon \Sh{K} \to L_{\dS{X}_{0}/\dS{X}}$ 
      is an equivalence of quasi-coherent $\Sh{O}_{\Sh{X}_0}$--modules 
      such that
      \[
        \xymatrix{
          L_{\Sh{O}_{\Sh{X}_0}} \ar[r]^{\delta} \ar[d]_d & \Sh{K} 
          \ar[dl]^{\widetilde{\delta}} \\
          L_{\Sh{O}_{\Sh{X}_0}/\Oof{X}}
        }
      \]
      commutes in $\QCoh (\dS{X}_0)$.
  \end{enumerate}
\end{defn}

We can now begin to prove geometric versions of our main result.  Since 
we want the objects inducing the obstruction theories to be of geometric 
nature, we have to make sure that in every step of the reconstruction we 
obtain geometric objects. The key to this is verifying that a 
square-zero extension of a connective spectral Deligne--Mumford stack by 
a quasi-coherent sheaf is again a spectral connective Deligne--Mumford 
stack.

\begin{lem}
  \label{lem:smallext}
  Let $n \geq 1$, and let $\dS{X}=\iT{X}$ be a connective spectral 
  Deligne--Mumford stack. Furthermore, let $\eta \colon L_{\Oof{X}} \to 
  \Sh{M}[1]$ be an $n$--connective derivation with $\Sh{M}$ a 
  quasi-coherent $\Oof{X}$--module.  Let $\Oof{X}^\eta \to \Oof{X}$ be 
  the corresponding extension under \fullref{thm:extension}. Then 
  $\dS{X}'=(\Sh{X},\Oof{X}^{\eta})$ is a connective spectral 
  Deligne--Mumford stack.
\end{lem}
\begin{proof}
  This is immediate by verify the conditions of \cite[Theorem  
  8.42]{DAGVII}. To verify the first condition, note that by the 
  assumption $n \geq 1$ the 0--truncations of $\dS{X}'$ and $\dS{X}$ are 
  equivalent.
\end{proof}

\begin{rem}
  Note that due to the connectivity assumption on $\Sh{M}$ in the above 
  lemma we have never changed the underlying $\infty$--topos, but only 
  have altered the structure sheaf.
\end{rem}

As in the algebraic case, we have to assume the existence of a global 
system of lifts of the obstruction theory.

\begin{defn}
  Let $\dS{X} = (\Sh{X},\Sh{O}_{\Sh{X}})$ be a spectral connective 
  Deligne--Mumford stack equipped with an $n$--connective quasi-coherent 
  obstruction theory $\phi \colon \Sh{E} \to L_{\Sh{O}_{\Sh{X}}}$, and 
  let $f \colon \dS{X} \to \dS{X}^{\eta}$ be the corresonding 
  square-zero extension. The obstruction \emph{lifts globally to 
    $\dS{X}^{\eta}$} if there exists an obstruction theory
  \[
    \phi^{\eta} \colon \Sh{E}^{\eta} \to L_{\Sh{O}_{\Sh{X}^{\eta}}}
  \]
  such that
  \[
    \xymatrix{
      f^* \Sh{E}^{\eta} \ar[r] \ar[d]_{\simeq} & 
      f^*L_{\Sh{O}_{\Sh{X}^{\eta}}} \ar[d]^{\simeq} \\
      \Sh{E} \ar[r]^{\phi'}&f^*L_{\Sh{O}_{\Sh{X}^{\eta}}}
    }
  \]
  commutes and the left vertical map is an equivalence. Here $\phi'$ is 
  induced by the fundamental diagram of the square-zero extension 
  $\dS{X}^{\eta} \to \dS{X}$ as in \eqref{eq:FundDiag}.
\end{defn}

\begin{defn}
  Let $\dS{X}_0 = (\Sh{X},\Sh{O}_{\Sh{X}_0})$ be a spectral connective 
  Deligne--Mumford stack equipped with an $n$--connective quasi-coherent 
  obstruction theory $\phi_0 \colon \Sh{E}_0 \to L_{\Sh{O}_{\Sh{X}_0}}$, 
  and let $\dS{X}_0 \to \dS{X}_1$ be the corresonding square-zero 
  extension.  A {global inductive system of lifts} of the obstruction 
  theory exists if for all $m \geq 0$ there exists a lift $\phi_{m+1} 
  \colon \Sh{E}_{m+1} \to L_{\Sh{O}_{\Sh{X}_{m+1}}}$ of the obstruction 
  theory $\phi_m$. Here $\Sh{X}_{m} \to \Sh{X}_{m+1}$ is the square-zero 
  extension defined by the obstruction theory $\phi_{m}$.
\end{defn}

Using the algebraic reconstruction theorem proven above, we can now 
prove a geometric version.

\begin{thm}
  \label{thm:geomresult}
  Let $\dS{X}_0 = (\Sh{X},\Sh{O}_{\Sh{X}_0})$ be a spectral connective 
  Deligne--Mumford stack equipped with an $n$--connective quasi-coherent 
  obstruction theory $\phi \colon \Sh{E} \to L_{\Sh{O}_{\Sh{X}_0}}$ with 
  $n \geq 1$.  Let $\Sh{K}=\cofib(\phi)$.
  Then a pair
  \[
    \left( i \colon \dS{X}_0 \to \dS{X},\widetilde{\delta} \colon \Sh{K} 
      \to L_{\Sh{O}_{\Sh{X}_0}/\Oof{X}} \right)
  \]
  inducing the obstruction theory exists if and only if a global system 
  of lifts of the obstruction theory exists.
\end{thm}
\begin{proof}
  First assume that a global system of lifts of the obstruction theory 
  exists. Let $\Cat{C}=\Shv_{\Sp}(\Sh{X})$, and using 
  \eqref{eq:calg-object} identify $\Sh{O}_{\Sh{X}_0}$ with an object of 
  $\CAlg(\Cat{C})$.  Applying \fullref{thm:algresult}, we obtain an 
  morphism $\Oof{X} \to \Sh{O}_{\Sh{X}_0}$ in $\CAlg(\Cat{C})$ and an 
  equivalence  $\widetilde{\delta} \colon \Sh{K} \to 
  L_{\Sh{O}_{\Sh{X}_0}/\Oof{X}}$
 in $\Mod_{ \Sh{O}_{\Sh{X}_0}}$. As every step of the construction given 
 in the proof of \fullref{thm:algresult} is a square-zero extension, the 
 pair $\iT{X}$ is indeed a spectral connective Deligne--Mumford stack by 
 \fullref{lem:smallext}. In particular, $L_{\Sh{O}_{\Sh{X}_0}/\Oof{X}}$ 
 is quasi-coherent. As $\QCoh(\dS{X}_0)$ is a full subcategory of 
 $\Mod_{ \Sh{O}_{\Sh{X}_0}}$, the claim follows.

 The proof is the converse is analogous to the algebraic case.
\end{proof}

We now proceed to prove a relative version of \fullref{thm:geomresult} 
over $\Spec (k)$. So let $p \colon \dS{X}_0 = (\Sh{X},\Sh{O}_{\Sh{X}_0}) 
\to \Spec(k)$ be a morphism of spectral connective Deligne--Mumford 
stacks.  Pulling back the structure sheaf of $\Spec (k)$, we obtain a 
morphism $p^* \Sh{O}_{\Spec (k)} \to \Sh{O}_{\Sh{X}_0}$ of connective 
commutative algebra objects in $\Shv_{\Sp}(\Sh{X})$. In particular we 
can view $\Sh{O}_{\Sh{X}_0}$ as a connective object of $\CAlg (\Mod_{p^* 
  \Sh{O}_{\Spec(k)}})$.

\begin{prop}
  \label{prop:geom_overk}
  Let $k$ be a connective $\IE_{\infty}$--ring, and let 
  $\dS{X}_0=(\Sh{X},\Sh{O}_{\Sh{X}_0})$ be a spectral connective 
  Deligne--Mumford stack over $\Spec (k)$ equipped with an $n$--connective 
  quasi-coherent obstruction theory $\phi \colon \Sh{E} \to 
  L_{\Sh{O}_{\Sh{X}_0}/k}$ with $n \geq 1$.  Then a pair
  \[
    \left( i \colon \dS{X}_0 \to \dS{X},\widetilde{\delta} \colon \Sh{K} 
      \to L_{\Sh{O}_{\Sh{X}_0}/\Oof{X}} \right)
  \]
  with $\dS{X}$ is a spectral connective Deligne--Mumford stack over 
  $\Spec (k)$ inducing the obstruction theory exists if and only if a 
  global inductive system of lifts of the obstruction theory 
  exists.
\end{prop}
\begin{proof}
  Assuming that a global system of lifts of the obstruction theory 
  exists, we apply \fullref{thm:algresult} to the category $\CAlg 
  (\Mod_{p^* \Sh{O}_{\Spec(k)}})$ to obtain a morphism of commutative 
  algebra objects $\Oof{X} \to \Sh{O}_{\Sh{X}_0}$ in $\CAlg (\Mod_{p^* 
    \Sh{O}_{\Spec(k)}})$ and an equivalence $\widetilde{\delta} \colon 
  \Sh{K} \to L_{\Sh{O}_{\Sh{X}_0}/\Oof{X}}$ of 
  $\Sh{O}_{\Sh{X}_0}$--modules.  Define $\dS{X}=\iT{X}$. The remainder 
  of the proof is analogous to the proof of \fullref{thm:geomresult}.  
\end{proof}

As in \fullref{prop:finpres} we want to ensure certain finiteness 
properties if we start with an ordinary Deligne--Mumford stack locally 
of finite presentation over $\pi_0 k$ equipped with an $n$--connective 
$m$--perfect obstruction theory.

\begin{lem}
  \label{lem:lfp}
  Let $\dS{X}_0$ be an ordinary Deligne--Mumford stack locally of finite 
  presentation over $\Spec (\pi_0 k)$ equipped with an $n$--connective 
  $m$--perfect obstruction theory $\phi \colon \Sh{E} \to 
  L_{\Sh{O}_{\Sh{X}_0}/k}$ with $n \geq 1$.  Then in any pair $ ( i 
  \colon \dS{X}_0 \to \dS{X},\widetilde{\delta} \colon \Sh{K}      \to 
  L_{\Sh{O}_{\Sh{X}_0}/\Oof{X}} )$ inducing the obstruction theory, the 
  connective spectral Deligne--Mumford stack $\dS{X}=\iT{X}$ is locally 
  of finite presentation over $\Spec (k)$ and the cotangent complex 
  $L_{\Oof{X}/k}$ is of Tor--amplitude $\leq m$.
\end{lem}
\begin{proof}
  As both assertions are local this follows from \fullref{prop:finpres}.
\end{proof}

In case $n=m$ and the Aityah class formalism is available we obtain the 
following geometric version of \fullref{cor:Atiyah}.

\begin{cor}
  \label{cor:qsmooth}
  Let $k$ be a discrete commutative ring $k$ containing the rationals 
  $\IQ$, so that the Atiyah class formalism is available. Then if 
  $\dS{X}_0$ is locally of finite presentation over $\Spec (k)$ and 
  equipped with an $n$--connective $n$--perfect obstruction theory for 
  $n \geq 1$,  a pair $ ( i \colon \dS{X}_0 \to  
  \dS{X},\widetilde{\delta} \colon \Sh{K} \to  
  L_{\Sh{O}_{\Sh{X}_0}/\Oof{X}} )$ inducing the obstruction theory 
  exists if and only if the inductively defined classes
  \[
    \Ext^0_{\Sh{O}_{\Sh{X}_n}}(\Sh{E}^{n},\Sh{K}_{n-1} \otimes 
    _{\Sh{O}_{\Sh{X}_n}} L_{\Sh{O}_{\Sh{X}_n}} )
  \]
  vanish.
  \end{cor}


\bibliographystyle{gtart}
\bibliography{bibliography}
\end{document}